\documentclass[11pt,a4paper]{article} 
\usepackage{epsf,epsfig,amsfonts,amsgen,amsmath,amstext,amsbsy,amsopn,amsthm}
\usepackage{amsmath}
\usepackage{amsfonts,amsthm,amssymb,bm}
\usepackage{amsfonts}
\usepackage{graphics}
\usepackage{latexsym,bm}
\usepackage{amsfonts,amsthm,amssymb,bbding}
\usepackage{indentfirst}
\usepackage{graphicx}
\usepackage{color}
\usepackage[colorlinks=true,anchorcolor=blue,filecolor=blue,linkcolor=red,urlcolor=blue,citecolor=blue]{hyperref}     
\usepackage{float}
\usepackage{tikz,enumerate}
\usepackage{geometry} 
\newtheorem{thm}{Theorem}[section]
\newtheorem{conj}[thm]{Conjecture}
\newtheorem{lemma}[thm]{Lemma}
\newtheorem{claim}{Claim}[section]

\newtheorem{defn}[thm]{Definition}

\renewcommand{\le}{\leqslant}
\renewcommand{\geq}{\geqslant}
\renewcommand{\ge}{\geqslant}

\begin{document}

\title{On a conjecture of distance spectral extremal problems}

\author{
Hongzhang Chen\thanks{School of Mathematics and Statistics, 
Gansu Center for Applied Mathematics, 
Lanzhou University, Lanzhou, Gansu, 730000, China. Email: \url{mnhzchern@gmail.com}.} 
\and 
Jianxi Li\thanks{School of Mathematics and Statistics, Minnan Normal University, Zhangzhou, Fujian, 363000, China. 
Email: \url{ptjxli@hotmail.com}. Partially supported by the NSF of Fujian Province (No. 2021J02048).}
\and 
Yongtao Li\thanks{Corresponding author. Yau Mathematical Sciences Center, Tsinghua University, Beijing, 100084, China. 
Email: \url{ytli0921@hnu.edu.cn}.}
}

\date{\today}
\maketitle

\begin{abstract}
Brualdi and Hoffman proposed a well-known problem of determining the graph with maximum adjacency spectral radius among all graphs with given size $m$. Early work by Friedland and Stanley addressed some specific cases.  
This problem was later completely solved by Rowlinson and recently revisited by Cheng and Weng. 
Pioneering work on the distance matrix was  carried out by Graham and Pollak, as well as by Graham and Lov\'{a}sz. 
The distance spectral radius $\rho (G)$ of a connected graph $G$ is the largest eigenvalue of its distance matrix. 

In this paper, we completely solve the problem of characterizing the connected graph with minimum distance spectral radius among all graphs with size $m$.  
Let $\mathcal{G}(m)$ be the class of connected graphs with $m$ edges.  
For every $m \ge 3$, let $n$ be the unique integer satisfying $\binom{n-1}{2} < m \le \binom{n}{2}$, and we write $m = \binom{n-1}{2} + s$ with $1 \le s \le n-1$. Recently, Lin and Zhou [\emph{Adv. in Appl. Math. 173 (2026)}] investigated the graph in $\mathcal{G}(m)$ that minimizes $\rho(G)$ in the range $s \ge \max\{ \frac{n-6}{2}, 1\}$. However, the problem is much more difficult in the remaining range $1 \le s \le \frac{n-7}{2}$, and they conjectured that the unique minimizer is $\overline{P_{n,s+1}}$, the complement of a balanced disjoint union of paths. Using novel matrix analysis, we solve this conjecture in the affirmative. Moreover, we provide a new unified proof for the entire range $1 \le s \le n-1$.

The key ingredients in our proof argument include an innovative comparison principle for the distance spectral radius, an increment analysis of $\Phi$-functions on paths and cycles, an argument for balancing path lengths, and a walk enumeration technique via the Neumann series.    
\end{abstract}



\section{Introduction} 
Let $G$ be a simple connected graph with vertex set $V(G)$ and edge set $E(G)$. 
We denote by $n$ and $e(G)$ the \emph{order} and \emph{size} of $G$, respectively.  
The complement of $G$ is denoted as $\overline G$. For any vertex $u\in V(G)$, let
$d_{G}(u)$ and $N_{G}(u)$ (or $d(u)$ and $N(u)$ for short) be the degree and the set of neighbors of $u$, respectively. Clearly, $d_{G}(u)=|N_{G}(u)|$. 
We write $\mathbf 1$ for the all-ones column vector and $J=\mathbf 1\mathbf 1^{\mathsf{T}}$ for the all-ones matrix. Let $I$ be the identity matrix. 
We denote by $C_{\ell}$ the cycle on $\ell$ vertices, and $P_k$ the path on $k$ vertices. 
The adjacency matrix of an $n$-vertex graph $G$ is defined as $A(G)=[a_{ij}]_{i,j=1}^n$, where $a_{ij}=1$ if $ij\in E(G)$, and $a_{ij=0}$ otherwise. The spectral radius $\lambda (G)$ of a graph $G$ is defined as the largest eigenvalue of its adjacency matrix $A(G)$.

\subsection{The Brualdi--Hoffman type extremal problem}

The study of bounding the spectral radius of a graph in terms of the number of edges has a long and rich history. It is well-known that $\lambda (G)< \sqrt{2m}$. Tracking back to 1985, 
 Brualdi and Hoffman \cite{BH1985} proved that if $G$ has $m\le {k \choose 2}$ edges for some integer $k\ge 2$, then $\lambda (G)\le k-1$. This bound was further extended by Friedland \cite{Fri1985}, and later by Stanley \cite{Sta1987} who showed that $\lambda (G)\le \frac{1}{2} (\sqrt{8m+1}-1)$, with equality if and only if $m={k \choose 2}$ and $G$ is a complete graph $K_k$.  Here, we ignore the possible isolated vertices if there are no confusions.  
  We write $m={s \choose 2} + t$ for some integers $s,t$ with $0\le t \le s-1$. 
Let $K_1 \vee_t K_s$ be the graph obtained from $K_s$ by adding a new vertex joining 
$t$ vertices of $K_s$. In 1988, Rowlinson \cite{Row1988} proved that if $G$ has $m$ edges, then 
$\lambda (G) \le \lambda (K_1\vee_t K_s)$, 
with equality if and only if $G=K_1\vee_t K_s$. 
This confirms a conjecture proposed by Brualdi and Hoffman \cite{BH1985}. Roughly speaking, 
the graph with the largest adjacency spectral radius packs almost all edges into the clique, and then attaches the remaining edges to that clique. 
We refer to \cite{CW2025} for a matrix realization of these spectral bounds.

\smallskip
Motivated by the Tur\'{a}n theorem, Nikiforov \cite{Niki2002cpc} 
 significantly improved the above bounds by proving that if $G$ is a $K_{r+1}$-free graph, then $\lambda^2 (G)\le (1-\frac{1}{r})2m$. In particular, this recovers a result of Nosal, saying that every $m$-edge triangle-free $G$ satisfies $\lambda (G)\le \sqrt{m}$. 
 For related results, we refer the readers to \cite{LZS2024,LZZ2025, LLZ2024-book-4-cycle}. 
 Recently, Li, Liu and Zhang \cite{LLZ-edge-spectral} proved that for any graph $F$ with chromatic number $r+1\ge 3$, if $G$ is an $F$-free graph with $m$ edges, then $\lambda^2 (G) \le \big(1- \frac{1}{r} +o(1) \big) 2m$. Extending Nikiforov's theorem, Li, Liu and Zhang  \cite{LLZ-edge-color-critical} proved that the $o(1)$ error term can be removed for color-critical graphs $F$; and $\lambda (G)\le \sqrt{m} + O(1)$ for almost-bipartite graphs $F$.

\subsection{The distance spectral extremal problem}

The distance between two vertices $u,v$, denoted by $d_G(u,v)$, is the length of a shortest path connecting them. The {\it distance matrix} of $G$ is defined as $D(G)=(d_G(u,v))_{u,v\in V(G)}$. 
The {\it distance spectral radius} $\rho(G)$ is defined as the largest eigenvalue of $D(G)$. 
Note that $D(G)$ is nonnegative and irreducible. By the Perron--Frobenius theorem, $\rho(G)$ is a positive simple eigenvalue, and there exists a unique positive unit eigenvector corresponding to $\rho (G)$. 
The study of eigenvalues of the distance matrix of a connected graph dates back to the classical work \cite{GP1971, GHH1977,GL1978}, where they provided an inequality related the number of negative distance eigenvalues
to the addressing problem in data communication system. 
The properties of distance matrix turn out to be significant in the study of the isometric embedding problems from a connected graph to the squashed hypercube; see \cite{WIS2012, Saw2016, Alon2021}. 
We refer to the comprehensive surveys~\cite{ah14,LSXZ2021} for recent development.

\smallskip
For connected graphs of fixed order $n$, the extremal problems are classical: Ruzieh and Powers \cite{RP1990} and Stevanovi\'c and Ili\'c \cite{SI2010} showed that the complete graph $K_n$ uniquely minimizes $\rho$, while the path $P_n$ uniquely maximizes $\rho$. 
In contrast, the adjacency spectral radius exhibits the opposite behavior: the complete graph $K_n$ is the unique maximizer, and the path $P_n$ is the unique minimizer. 
Here, we provide an explanation for this reversed phenomenon. 
The adjacency spectral radius $\lambda (G)$ measures how well a graph expands or how strongly it is connected: a large value indicates a dense  graph with many edges and small diameter (e.g., a clique). In contrast, the distance spectral radius $\rho(G)$ measures the overall separation among vertices: a large value means many pairs of vertices are far apart, so the graph is sparse (e.g., a path or a star). Hence, dense graphs tend to have large $\lambda (G)$ but small $\rho(G)$, while sparse graphs tend to have small $\lambda (G)$ but large $\rho(G)$.

\smallskip 
A much harder problem is to find the extremal graphs among all connected graphs with a prescribed number of edges $m$, which was initiated by Lin and Zhou~\cite{LZ2026}. Let $\mathcal G(m)$ denote the class of such graphs. Throughout this paper, for a given $m \ge 3$, we denote 
\begin{equation} \label{eq-n-s}
n :=\Bigl\lceil \tfrac{1+\sqrt{8m+1}}{2}\Bigr\rceil, \qquad s:=m-\binom{n-1}{2}\in\{1,2,\ldots,n-1\}.
\end{equation}
For positive integers $N\ge c\ge 1$, we define  
$$
P_{N,c} \, :=\;\bigl(c+c\lfloor N/c\rfloor - N\bigr)\,P_{\lfloor N/c\rfloor}\;\cup\;\bigl(N-c\lfloor N/c\rfloor\bigr)\,P_{\lceil N/c\rceil},
$$
the {\it balanced disjoint union} of $c$ paths on $N$ vertices (its parts differ in order by at most one). 
By the definition, it is easy to verify that $\overline{P_{n,s+1}}$ contains 
${n \choose 2} - (n-s-1) = {n-1 \choose 2} +s =m$ edges.  Intuitively, we see that $\overline{P_{n,s+1}}$ is close to almost-complete graph, which is obtained from the largest possible  $K_n$ by deleting the edges of those $s+1$ balanced vertex-disjoint paths.

\smallskip
Recently, Lin and Zhou~\cite[Theorem 1.3]{LZ2026} proved that if $\max\{ \frac{n-6}{2},1\}\le s\le n-1$, then $\overline{P_{n,s+1}}$ is the unique graph that minimizes the distance spectral radius over all graphs in $\mathcal G(m)$. In the concluding remarks of~\cite{LZ2026}, they conjectured that the same formula holds for all small $s$. 

\begin{conj}[Lin--Zhou \cite{LZ2026}]\label{conj-LZ}
Let $G$ be a graph with $m$ edges, and let $n$ and $s$ be defined as in (\ref{eq-n-s}). If
$1\le s\le \frac{n-7}{2}$, then 
$
\rho(G)\ge\rho\bigl(\overline{P_{n,s+1}}\bigr),
$ 
with equality if and only if $G\cong \overline{P_{n,s+1}}$.
\end{conj}

In this paper, we resolve Conjecture \ref{conj-LZ} by showing the following result.

\begin{thm}\label{main}
Let $G\in\mathcal G(m)$ minimize $\rho$ of graphs  in $\mathcal G(m)$, and let $n=\lceil \frac{1}{2}(1+\sqrt{8m+1})\rceil$ and $s=m-\binom{n-1}{2}$. 
If $1\le s < \frac{n-2}{2}$, then $G\cong \overline{P_{n,s+1}}$. In particular, Conjecture \ref{conj-LZ} holds.
\end{thm}

\noindent
{\bf Remark.} In the case $s=0$, the above result is invalid. In this setting, 
we have $m={n-1 \choose 2}$, the minimizer is uniquely the complete graph $K_{n-1}$, instead of the expected graph $\overline{P_{n,1}}=\overline{P_n}$.  

\smallskip
The range $\max\{\frac{n-6}{2},1\}\le  s \le n-1$ was already investigated by Lin and Zhou \cite{LZ2026}, which together with the range in Theorem \ref{main} yields the following {\it complete solution} of the distance spectral extremal problem for graphs with given size $m$. 

\begin{thm}\label{t-1}
For every integer $m\ge 3$, let $n=\lceil \frac{1}{2}(1+\sqrt{8m+1})\rceil$ and $s=m-\binom{n-1}{2}$. 
Then for all $1\le s \le n-1$, the graph $\overline{P_{n,s+1}}$ uniquely minimizes $\rho$ over all graphs in $\mathcal{G}(m)$. 
\end{thm}

As an application of our framework,  we shall provide {\it a new self-contained proof} for  $s\ge \frac{n-2}{2}$ in Section \ref{sec-short-proof}, thereby giving a unified treatment of Theorem \ref{t-1} for the entire range $1\le s\le n-1$.

\smallskip
Before the proof, we first illustrate the difficulty for the small $s$. 
To appreciate why the range of small $s$ was much difficult and left open, we now  recall how the proof argument of \cite[Theorem 1.3]{LZ2026} worked for large $s$. 
A structural lemma (we state it as in Theorem~\ref{thm:structure} below) shows that any minimizer $G$ has complement 
$
\overline G= C_{\ell_1}\cup\cdots\cup C_{\ell_t}\cup P_{k_1}\cup\cdots\cup P_{k_{s+1}}
$, where $\ell_i\ge 3$, $k_j\ge 2$ 
and the number of path components equals $s+1$. The problem is therefore to show that the extremal graph is $t=0$ together with the balanced path partition $P_{n,s+1}$. 
Lin and Zhou's argument proceeds by local edge-switching: given two candidate configurations, one exchanges a single edge and uses a Rayleigh quotient to show $\rho$ strictly decreases. For large $s$, this works well because the number of candidate configurations is small and can be enumerated. In fact, when $s=\lceil \frac{1}{2}(n-5)\rceil$, the only cyclic candidate is $C_3\cup\tfrac{n-3}{2}K_2$, and a single edge-switch to $P_{n,(n-3)/2}$ suffices.

\smallskip
However, when $s$ is small, the number of possible cycle structures grows without bound, and as one can verify numerically, no single step edge-switch is monotone. There exist cyclic configurations $G$ such that every local rearrangement, which breaks one cycle and reconnects one path, either fails to be admissible or increases $\rho$. The following  example with $n=11$ and $s=1$ illustrates this. 
$$
\rho(\overline{C_3\cup P_4\cup P_4})\approx 11.65444,\qquad \rho(\overline{P_3\cup P_8})\approx 11.65452.
$$ 
The natural operation, which breaks the $C_3$ into $P_3$ and merges the two $P_4$'s into $P_8$, increases the value of $\rho$, even though the minimum $\rho(\overline{P_{11,2}})=\rho(\overline{P_5\cup P_6})\approx 11.65442$ is strictly smaller than both of the above two graphs. 
In other words, the minimum can {\it not} be reached by a monotone sequence of single-step improvements of the edge-switch operation. To overcome the difficulty, a global argument to compare the distance spectral radii of graphs is required.

\paragraph{Our approach.} 
We get rid of the use of local edge-switching operations as it fails to be monotone for small integer $s$. Instead, we develop a new global comparison framework grounded from matrix analysis. 
Our approach can be partitioned into three stages. 

The first stage lies in linking the distance matrix $D(G)$ to the adjacency matrix $A(H_0)$ of the complement $H_0=\overline{G}$. We show an identity $D(G) = J - I + A(H_0)$, and that $\rho (G)$ is the unique root of the equation $\mathbf 1^{\mathsf{T}} ((\rho+1)I - A(H_0) )^{-1}\mathbf 1 =1$; see Lemma \ref{lem-2.2}. 
For a graph $K$, we denote 
$\Phi_K(\lambda) :=  \bm 1^{\mathsf T} (\lambda I - A(K))^{-1} \bm 1$. 
Let $\mathcal{C}(H_0)$ be the set of components of $H_0$. We define the $\Psi$-function of $G$ as $\Psi_G(\lambda)  := \sum_{K \in \mathcal{C}(H_0)} \Phi_K(\lambda)$; see Definition \ref{defn}.  
Under the above notation, we reduce the problem of comparing distance spectral radii to evaluating these $\Psi$-functions; see Theorem \ref{thm-comparison}.  

The second stage is to compute the explicit formula of $\Phi_K(\lambda)$ when $K$ is chosen as a path and cycle; see Lemma \ref{lem-Phi-cycle-path}. Based on these formulas, we then perform a rigorous increment analysis of $\Phi_K(\lambda)$ when $K$ is a path. We provide an upper bound on the increment and establish a strict convexity property, showing that the map $k\mapsto \Phi_{P_k}(\lambda)$ is strictly convex on $k$; see Theorem \ref{thm-increment}. 

The third stage is an argument for balancing path lengths. 
The properties from the second stage allow us to formulate balancing principles, demonstrating 
that the complement $H_0$ of the extremal graph must consist exclusively of paths of nearly equal lengths; see Lemma \ref{lem-5.1} and Claim \ref{claim:1}. Moreover, we show that $H_0$ must be devoid of any cycles; see Lemma \ref{lem-5.2} and Claim \ref{claim:2}.    

Finally, we present a short self-contained proof in the  range $s\ge \frac{n-2}{2}$. 
By expanding the $\Psi$-function via the Neumann series, we see that 
$\Psi_G(\lambda) = \sum_{k=0}^{\infty} \frac{w_k(H_0)}{\lambda^{k+1}}$, 
where $w_k(H_0)$ denotes the total number of walks of length $k$ in $H_0$, we integrate the comparison principle in Theorem \ref{thm-comparison} with walk enumeration to provide {\it a new unified proof} covering the entire range $1 \le s \le n-1$. 

We anticipate that the method developed in this paper will offer an efficient approach to solving additional extremal problems related to graphs and their distance spectral radii. 

\section{Proof of the main result}

In order to prove Theorem \ref{main}, 
we need to present some preliminary results.
In the range $1\le s< \frac{n-2}{2}$,  
the expected extremal graphs in $\mathcal G(m)$ are restricted by the following result.

\begin{thm}[Lin--Zhou \cite{LZ2026}]\label{thm:structure}
Let $G\in\mathcal G(m)$ be a graph that minimizes $\rho$ of graphs in $\mathcal G(m)$. 
Then $|V(G)|=n$. In addition, if $1\le s < \frac{n-2}{2}$, then each component of $\overline G$ is either a cycle or a non-trivial path, and there are exactly $s+1$ components being paths in $\overline G$.
\end{thm}

Throughout the rest of the paper, we denote by 
$$
H_0 =\overline G= C_{\ell_1}\cup \cdots\cup C_{\ell_t}\;\cup\; P_{k_1}\cup \cdots\cup P_{k_{s+1}},
$$
where $t\ge 0,\  \ell_i\ge 3,\ k_j\ge 2$ satisfying  
$N_c:=\sum_{i=1}^t \ell_i $ and $N_p:=\sum_{j=1}^{s+1} k_j =n-N_c.$
Based on this result, our goal is to prove that $t=0$ (no cycles) and the path lengths $k_1, k_2,\ldots ,k_{s+1}$ are as equal as possible. We point out that this is not as easy as it looks. To achieve this goal, we need to establish a series of important lemmas and theorems. 
We will now present them one by one.

\smallskip 
The following lemma from Matrix Analysis can be found in \cite[page 7]{Bha1997}. 

\begin{lemma}[Neumann series, e.g., \cite{Bha1997}] \label{lem-Neumann}
    Let \(A \in \mathbb{C}^{n \times n}\) be a square matrix.
    If \(\|A\| < 1\) for some matrix norm (or equivalently,
    the spectral radius \(\lambda (A) < 1\)), then \(I - A\) is invertible and
    \[
    (I - A)^{-1} = \sum_{i=0}^{\infty} A^i = I + A + A^2 + A^3 + \cdots
    \]
    where the series converges in any matrix norm.
\end{lemma}

\subsection{New comparison principle via matrix analysis}

The next lemma is the cornerstone of our approach. 
We establish an interesting relation between the distance spectral radius $\rho (G)$ and the adjacency matrix of $H_0$. 
The diameter of a connected graph $G$ is defined to be the maximum distance between any two vertices in $G$. 

\begin{lemma}\label{lem-2.2}
Let $H_0$ be a disjoint union of cycles and paths, with at least two components, and let $G=\overline{H_0}$ on $n\ge 4$ vertices. Let $A(H_0)$ be the adjacency matrix of $H_0$. 
Then
\begin{enumerate}[(i)]
\item $G$ is connected with diameter exactly $2$, 
and $D(G)=J-I+A(H_0)$.
\item $\rho:=\rho(G)$ is the unique root in $(2,\infty)$ of the equation 
\begin{equation}\label{e-1}
\mathbf 1^{\mathsf{T}}\bigl((\rho+1)I - A(H_0)\bigr)^{-1}\mathbf 1 =1. 
\end{equation}
\item Every entry of $\bigl(\lambda I - A(H_0)\bigr)^{-1}$ 
is positive and strictly decreasing in $\lambda$ on $(2,\infty)$. In particular, the left-hand side of (\ref{e-1}) is strictly decreasing in $\rho$.
\end{enumerate}
\end{lemma}

\begin{proof}[{\bf Proof of (i)}] 
By definition, it is easy to see that $G$ is a connected graph with diameter $2$. 
For any two distinct $u,v\in V(G)$, we see that  $d_{G}(u,v)=1$ if and only if $uv\notin E(H_0)$;  and $d_{G}(u,v)=2$ if and only if $uv\in E(H_0)$. Thus, it follows that 
$$
d_G(u,v)=1+\mathbf 1_{\{uv\in E(H_0)\}}\;=\;(J-I)_{uv}+(A(H_0))_{uv}. 
$$
For $u=v$, both sides are $0$. Hence, we obtain  $D(G)=J-I+A(H_0)$, as needed. 

\vspace{2mm}
\noindent{\bf Proof of (ii).} Since $D(G)$ is irreducible and nonnegative, the Perron--Frobenius theorem yields a strictly positive eigenvector $\bm x>0$ associated with $\rho$. Normalize so that $\mathbf 1^{\mathsf{T}}\bm x=1$. From part (i), 
$$
\rho\, \bm x=D(G)\bm x= \big(J-I+A(H_0) \big)\bm x=\mathbf 1(\mathbf 1^{\mathsf{T}}\bm x)-\bm x+A(H_0)\bm x=\mathbf 1-\bm x+A(H_0)\bm x.
$$
Rearranging the above gives 
\begin{equation}\label{e-2}
\bigl((\rho+1)I - A(H_0)\bigr)\bm x = \mathbf 1. 
\end{equation}

We next show that the matrix $(\rho+1)I-A(H_0)$ is invertible. Since every component $K$ of $H_0$ is a cycle or a path, the spectral radius  $\lambda(H_0)\le 2$. 
Since each vertex has $\deg_{H_0}(u)\ge 1$, we have $\deg_G(u)\le n-2$. 
On the other hand, 
since $G$ has diameter $2$, every row sum of $D(G)$ equals to $2(n-1)-\deg_G(u) \ge n$. 
By the Perron–Frobenius theorem, we have $\rho(G)\ge  n\ge 4$. Therefore $\rho+1\ge n+1>2\ge \lambda(H_0)$, and $(\rho+1)I-A(H_0)$ is invertible. Solving (\ref{e-2}) yields 
$$
\bm{x} = \big((\rho+1)I-A(H_0) \big)^{-1}\mathbf 1.
$$ 
multiplying $\mathbf 1^{\mathsf{T}}$ of both sides and using the assumption $\mathbf 1^{\mathsf{T}} \bm{x}=1$ gives the desired identity (\ref{e-1}).

\vspace{2mm}
\noindent{\bf Proof of (iii).} For $\lambda>2\geq \lambda(H_0)$, the Neumann series in Lemma \ref{lem-Neumann} gives  
\begin{equation*} 
\big(\lambda I- A(H_0) \big)^{-1} 
=\sum_{i=0}^{\infty}\frac{A(H_0)^i}{\lambda^{i+1}}, 
\end{equation*} 
where the series converges entrywise. Since $A(H_0)$ is nonnegative, 
each entry of the partial sums is a nonnegative, strictly decreasing function of $\lambda$, and the same holds in the limit. In particular, the left-hand side of (\ref{e-1}) equals $\sum_{i= 0}^{\infty} \lambda^{-(i+1)}\mathbf 1^{\mathsf{T}}A(H_0)^i \mathbf 1$, whose $i=0$ term $\frac{n}{\lambda}$ is already strictly decreasing, so the whole sum is strictly decreasing in $\lambda$ on $(2,\infty)$. This proves (iii), and together with the identity (\ref{e-1}) at $\lambda=\rho+1$, also 
show that (\ref{e-1}) has a unique root in $(2,\infty)$.
\end{proof}

 For each component $K$ of $H_0$ and $\lambda>2$, we define 
\begin{equation*}
\Phi_K(\lambda):=\mathbf 1^{\mathsf{T}}\bigl(\lambda I - A(K)\bigr)^{-1}\mathbf 1. 
\end{equation*} 

Next, we extend the notation to the graph $H_0$.  
Since $A(H_0)$ is block-diagonal with each block corresponding to a component of $H_0$,  its inverse matrix $(\lambda I-A(H_0))^{-1}$ is also  block-diagonal. 

\begin{defn} \label{defn}
Let $H_0=\overline{G}$ be a disjoint union of cycles and paths, with at least two components.  
Denote $\mathcal{C}(H_0)$ the set of components of $H_0 $. We define {\it the $\Psi$-function} of $G$ as follows:   
\begin{equation*}
\Psi_{G}(\lambda):= \Phi_{H_0}(\lambda) = \sum_{K\in \mathcal{C}(H_0)} \Phi_K(\lambda). 
\end{equation*}
\end{defn}

Thus, Lemma \ref{lem-2.2} (ii) gives $\Psi_{G}(\rho(G) +1)=1$. 
Lemma~\ref{lem-2.2} (iii) implies that $\Phi_{H_0}(\lambda)$ is strictly decreasing in $\lambda \in (2,\infty)$. 
By Definition \ref{defn}, we know that $\Psi_{G}(\lambda)$ 
is also strictly decreasing in $\lambda$. 
Comparing the distance spectral radius of two graphs reduces to comparing their $\Psi$-functions. 
The following theorem can be viewed as a key innovation in this paper to solve Conjecture \ref{conj-LZ}. 

\begin{thm}\label{thm-comparison}
Let $G_1,G_2$ be two graphs satisfying the hypotheses of Lemma~\ref{lem-2.2}, with corresponding $\Psi$-functions $\Psi_{1},\Psi_{2}$, respectively. If $\Psi_1(\lambda)<\Psi_2(\lambda)$ for every $\lambda>2$, then $\rho(G_1)<\rho(G_2)$. 
More generally, if the strict inequality holds merely at 
$\lambda=\rho(G_2)+1$, then $\rho(G_1)<\rho(G_2)$.
\end{thm}

\begin{proof}
Setting $\lambda_i := \rho(G_i)+1$, we have $\Psi_{i}(\lambda_i)=1$. 
The assumption implies 
\[ \Psi_{1}(\lambda_2)< \Psi_{2}(\lambda_2)= 1= \Psi_{1}(\lambda_1). \] 
Since $\Psi_{1}(\lambda)$ is strictly decreasing by Lemma \ref{lem-2.2}, we have $\lambda_2>\lambda_1$, i.e., $\rho(G_2)>\rho(G_1)$. 
\end{proof}

\noindent
{\bf Remark.} 
By applying Theorem \ref{thm-comparison}, we no longer compare the increment of the distance spectral radius via the entries of Perron eigenvector and the Rayleigh quotient of $\rho (G)$.

\subsection{The increment of $\Phi$-functions on paths and cycles}

We now compute $\Phi_K(\lambda)$ in closed form for the two types of components relevant here.

\begin{lemma}\label{lem-Phi-cycle-path}
Let $\lambda>2$ and denote  $\theta:=\tfrac12\bigl(\lambda+\sqrt{\lambda^2-4}\bigr)>1$, so that $\lambda=\theta+\theta^{-1}$. 
\begin{enumerate}[(i)]
\item For any cycle $C_\ell$ with $\ell\ge 3$, we have 
\begin{equation*} 
\Phi_{C_\ell}(\lambda)\;=\;\frac{\ell}{\lambda-2}. 
\end{equation*}

\item For any path $P_k$ with $k\ge 2$, we have 
\begin{equation*}
\Phi_{P_k}(\lambda)\;=\;\frac{k}{\lambda-2} 
- \frac{2}{\lambda -2} R_k(\theta),\quad \text{where} \quad R_k(\theta):=\frac{1}{\theta-1}\!\left(1-\frac{1+\theta}{1+\theta^{k+1}}\right).
\end{equation*}
\end{enumerate}
\end{lemma}

\begin{proof}[{\bf Proof of (i)}]
Since $C_\ell$ is $2$-regular, we have $A(C_\ell)\mathbf 1=2\mathbf 1$. Hence 
$$
(\lambda I - A(C_\ell))\mathbf 1=(\lambda-2)\mathbf 1.
$$
Since $\lambda>2$, the matrix $\lambda I - A(C_\ell)$ is invertible. So 
$$
(\lambda I - A(C_\ell))^{-1}\mathbf 1=\frac{1}{\lambda-2}\mathbf 1.
$$
Recall that $\mathbf 1^{\mathsf{T}}\mathbf 1=\ell$. We obtain $\Phi_{C_{\ell}}(\lambda) =\mathbf 1^{\mathsf{T}}\bigl(\lambda I - A(C_{\ell})\bigr)^{-1}\mathbf 1 =\frac{\ell}{\lambda -2}$. 

\vspace{2mm}
\noindent{\bf Proof of (ii).} 
Suppose that $\bm{y}=(y_1,\ldots ,y_k)^{\mathsf{T}}$ is a vector satisfying  $(\lambda I-A(P_k)) \bm{y}=\mathbf 1$. We need to compute $\mathbf 1^{\mathsf{T}} \bm{y} =\sum_{i=1}^k y_i$. 
For $2\le i\le k-1$, the equation $(\lambda I-A(P_k)) \bm{y} =\mathbf 1$ at $i$-th row reads
$$
-y_{i-1}+\lambda y_i - y_{i+1}=1.
$$
At the first and last rows ($i=1$ and $i=k$), the equations  gives 
$$
\lambda y_1 - y_2=1,\qquad -y_{k-1}+\lambda y_k=1.
$$
Introducing auxiliary values $y_0:=0$ and $y_{k+1}:=0$, so $y_0,y_1,\ldots,y_k,y_{k+1}$ satisfy 
\begin{equation}\label{e-3.3}
-y_{i-1}+\lambda y_i-y_{i+1}=1\qquad (1\le i\le k).  
\end{equation}

The characteristic equation of the recurrence $-y_{i-1}+\lambda y_i-y_{i+1}=0$ is given as $x^2-\lambda x +1=0$, which has two real roots $\theta$ and $\theta^{-1}$, where $\theta =\tfrac12\bigl(\lambda+\sqrt{\lambda^2-4}\bigr)$. 
A particular solution by setting $y_i\equiv c$, gives $(-1+\lambda-1)c=1$, i.e., $c=\frac{1}{\lambda-2}$. 
Thus, the general solution of (\ref{e-3.3}) is given as 
\begin{equation}\label{e-3.4}
y_i=\frac{1}{\lambda-2}+C_1\theta^i+C_2\theta^{-i}, \quad  (0\le i\le k+1)
\end{equation}
where $C_1,C_2\in\mathbb R$ are some constants determined later.

The boundary condition $y_0=0$ gives 
\begin{equation}\label{e-3.5}
0=\frac{1}{\lambda-2}+C_1+C_2.
\end{equation}
The condition $y_{k+1}=0$ gives
\begin{equation}\label{e-3.6}
0=\frac{1}{\lambda-2}+C_1\theta^{k+1}+C_2\theta^{-(k+1)}. 
\end{equation}
From (\ref{e-3.5}), we have $C_2=-\tfrac{1}{\lambda-2}-C_1$. Substituting into (\ref{e-3.6}) yields 
$$
0=\frac{1}{\lambda-2}+C_1\theta^{k+1}+\Bigl(-\frac{1}{\lambda-2}-C_1\Bigr)\theta^{-(k+1)}.
$$
Collecting terms, we have 
$$
C_1\bigl(\theta^{k+1}-\theta^{-(k+1)}\bigr)=-\frac{1}{\lambda-2}\bigl(1-\theta^{-(k+1)}\bigr).
$$
Multiplying by $\theta^{k+1}$ both sides, we obtain  
$C_1 {(\theta^{2(k+1)}-1)} 
=-\frac{1}{\lambda-2} {(\theta^{k+1}-1)} $, 
so
\begin{equation}\label{e-3.7}
C_1=-\frac{1}{\lambda-2}\cdot\frac{\theta^{k+1}-1}{\theta^{2(k+1)}-1} =\frac{-1}{(\lambda-2)(1+\theta^{k+1})}. 
\end{equation}
Substituting back into (\ref{e-3.5}), it follows that 
\begin{equation}\label{e-3.8}
C_2=-\frac{1}{\lambda-2}-C_1=\frac{-(1+\theta^{k+1})+1}{(\lambda-2)(1+\theta^{k+1})}=\frac{-\theta^{k+1}}{(\lambda-2)(1+\theta^{k+1})}.
\end{equation}
Combining with (\ref{e-3.4}), we conclude that 
$$
\sum_{i=1}^k y_i = \frac{k}{\lambda-2}+C_1\,S_1+C_2\,S_2,
$$
where $S_1:=\sum_{i=1}^k \theta^i=\frac{\theta(\theta^k-1)}{\theta-1}$ and $S_2:=\sum_{i=1}^k \theta^{-i}=\frac{\theta^k-1}{\theta^k(\theta-1)}$. 

Using (\ref{e-3.7}) and (\ref{e-3.8}), we have 
$$
\begin{aligned}
C_1 S_1+C_2 S_2
&=\frac{-(\theta^k-1)}{(\lambda-2)(\theta-1)(1+\theta^{k+1})}\cdot\Bigl(\theta+\frac{\theta^{k+1}}{\theta^k}\Bigr)\\
&=\frac{-(\theta^k-1)}{(\lambda-2)(\theta-1)(1+\theta^{k+1})}\cdot 2\theta 
= \frac{-2}{\lambda-2} \cdot 
\frac{\theta(\theta^k-1)}{(\theta-1)(1+\theta^{k+1})}.
\end{aligned}
$$ 
We now show that
\begin{equation*} 
\frac{\theta(\theta^k-1)}{(\theta-1)(1+\theta^{k+1})}\;=\;R_k(\theta).
\end{equation*} 
This can be verified by rewriting the left-hand side as 
\begin{align*}
\frac{\theta(\theta^k-1)}{(\theta-1)(1+\theta^{k+1})} =\frac{1}{\theta-1}\cdot\frac{(1+\theta^{k+1})-(1+\theta)}{1+\theta^{k+1}}=\frac{1}{\theta-1}\!\left(1-\frac{1+\theta}{1+\theta^{k+1}}\right)= R_k(\theta). 
\end{align*}
Putting all the above identities together, we conclude that 
$$
\Phi_{P_k}(\lambda) 
= \bm{1}^{\mathsf T} \big(\lambda I - A(P_k) \big)^{-1} \bm{1} = \bm{1}^{\mathsf T} \cdot \bm{y} 
=\sum_{i=1}^k y_i=\frac{k}{\lambda-2}-\frac{2}{\lambda-2}R_k(\theta).
$$
This proves the desired identity in part (ii).
\end{proof}

This section contains the main technical result of the paper.

\begin{thm}\label{thm-increment}
For every $\lambda>2$ and $k\ge 2$, let $\theta>1$ be the root of $\theta+\theta^{-1}=\lambda$. Then 
\begin{equation}\label{e-4.1}
\Phi_{P_{k+1}}(\lambda)-\Phi_{P_k}(\lambda)\;=\;\frac{1}{\lambda-2}\cdot\frac{(\theta^{k+1}-1)(\theta^{k+2}-1)}{(1+\theta^{k+1})(1+\theta^{k+2})}.
\end{equation}
Consequently, we obtain the following two statements: 
\begin{enumerate}[(i)]
\item For all $k\ge 2$ and $m\ge 1$, we have 
$0<\Phi_{P_{k+m}}(\lambda)-\Phi_{P_k}(\lambda)<\frac{m}{\lambda-2}$. 

\item The map $k\mapsto \Phi_{P_k}(\lambda)$ is strictly convex on $\{k\in \mathbb{N}: k\ge 2\}$.
\end{enumerate}
\end{thm}

\begin{proof}
Applying Lemma \ref{lem-Phi-cycle-path} (ii), it follows that  
\begin{equation}\label{e-4.3}
\Phi_{P_{k+1}}(\lambda)-\Phi_{P_k}(\lambda) 
=\frac{1}{\lambda -2} - \frac{2}{\lambda-2} \big(R_{k+1} (\theta )-R_k (\theta ) \big). 
\end{equation}
We denote 
$ g(j):=\frac{1}{1+\theta^j}$. 
By definition, we get $R_k (\theta )=\tfrac{1}{\theta-1}\bigl(1-(1+\theta)g(k+1)\bigr)$, so
\begin{equation}\label{e-4.4}
R_{k+1} (\theta ) -R_k (\theta ) =-\frac{1+\theta}{\theta-1}\bigl(g(k+2)-g(k+1)\bigr).
\end{equation}
We now compute the difference $g(j+1)-g(j)$ and obtain 
$$
g(j+1)-g(j)=\frac{1}{1+\theta^{j+1}}-\frac{1}{1+\theta^j} =\frac{\theta^j-\theta^{j+1}}{(1+\theta^j)(1+\theta^{j+1})}=-\frac{\theta^j(\theta-1)}{(1+\theta^j)(1+\theta^{j+1})}.
$$
Replacing with $j=k+1$, we get 
$$
g(k+2)-g(k+1)=-\frac{\theta^{k+1}(\theta-1)}{(1+\theta^{k+1})(1+\theta^{k+2})}.
$$
Substituting into (\ref{e-4.4}), we obtain 
\begin{equation*} 
R_{k+1} (\theta ) -R_k (\theta )=-\frac{1+\theta}{\theta-1}\cdot\left(-\frac{\theta^{k+1}(\theta-1)}{(1+\theta^{k+1})(1+\theta^{k+2})}\right)=\frac{(1+\theta)\,\theta^{k+1}}{(1+\theta^{k+1})(1+\theta^{k+2})}.    
\end{equation*}
Now we need to compute $1-2(R_{k+1} (\theta )-R_k (\theta ))$. Note that 
$$ 
\begin{aligned}
1-2(R_{k+1} (\theta )-R_k (\theta ) ) =1-\frac{2(1+\theta)\theta^{k+1}}{(1+\theta^{k+1})(1+\theta^{k+2})} 
=\frac{(1+\theta^{k+1})(1+\theta^{k+2})-2(1+\theta)\theta^{k+1}}{(1+\theta^{k+1})(1+\theta^{k+2})},
\end{aligned}
$$
where the numerator can be expanded as 
$$
1+\theta^{k+1}+\theta^{k+2}+\theta^{2k+3} 
- 2\theta^{k+1}-2\theta^{k+2} 
=1-\theta^{k+1}-\theta^{k+2}+\theta^{2k+3} 
= (\theta^{k+1}-1)(\theta^{k+2}-1).
$$
Therefore, we conclude that 
\begin{equation}\label{e-4.6}
1-2(R_{k+1}-R_k)=\frac{(\theta^{k+1}-1)(\theta^{k+2}-1)}{(1+\theta^{k+1})(1+\theta^{k+2})}.
\end{equation}
Combining (\ref{e-4.6}) and (\ref{e-4.3}) yields the expected identity (\ref{e-4.1}).

\vspace{2mm}
\noindent{\bf Proof of (i).} 
For every $\theta>1$ and $j\ge 1$, we have 
$0<\frac{\theta^j-1}{1+\theta^j}<1$. 
Then 
$$
0\;<\;\frac{\theta^{k+1}-1}{1+\theta^{k+1}} \cdot 
\frac{\theta^{k+2}-1}{1+\theta^{k+2}}
\;<\;1.
$$
Combining this with (\ref{e-4.1}), we have 
$$
0<\Phi_{P_{k+1}}(\lambda)-\Phi_{P_k}(\lambda)<\frac{1}{\lambda-2}.
$$
Telescoping the above inequality from $k$ to $k+m$, we get 
$$
\Phi_{P_{k+m}}(\lambda)-\Phi_{P_k}(\lambda)=\sum_{j=0}^{m-1}\bigl(\Phi_{P_{k+j+1}}(\lambda)-\Phi_{P_{k+j}}(\lambda)\bigr) <
\frac{m}{\lambda-2}.
$$

\vspace{2mm}
\noindent{\bf Proof of (ii).} 
Using (\ref{e-4.1}), we denote the increment 
\begin{equation} \label{eq-Delta}
\Delta \Phi_k:=\Phi_{P_{k+1}}(\lambda)-\Phi_{P_k}(\lambda)=\frac{h(k)}{\lambda-2}, 
\end{equation}
where 
$$
h(k):=\frac{(\theta^{k+1}-1)(\theta^{k+2}-1)}{(1+\theta^{k+1})(1+\theta^{k+2})}=\prod_{j=1}^2 \frac{\theta^{k+j}-1}{\theta^{k+j}+1}= \prod_{j=1}^2 \left(1-\frac{2}{\theta^{k+j}+1}\right).
$$
For each $j\in\{1,2\}$, since $\theta>1$ makes $\theta^{k+j}$ strictly increasing in $k$, and $2/(\theta^{k+j}+1)$ strictly decreasing, then the factor $1-\frac{2}{\theta^{k+j}+1}$ is strictly increasing in $k$. Hence $h(k)$ is strictly increasing in $k$, giving $\Delta \Phi_{k+1}>\Delta \Phi_{k}$. 
Consequently, we obtain the convexity 
$\Phi_{P_{k+2}} (\lambda) +\Phi_{P_{k}} (\lambda) >2\Phi_{P_{k+1}} (\lambda)$.
\end{proof}

\subsection{Balancing principles}

We now derive two results: first, for a fixed total order and fixed number of parts, path lengths should be as balanced as possible (Lemma~\ref{lem-5.1}); second, enlarging the total order monotonically enlarges each part of the balanced partition (Lemma~\ref{lem-5.2}).

\begin{lemma}\label{lem-5.1}
Let $c\ge 1$ and $N\ge 2c$ be integers, and let $k_1,\ldots,k_c\ge 2$ be integers with $\sum_{i=1}^c k_i =N$. We denote $(k_i^{\ast})_{i=1}^c$ the balanced partition of $N$ into $c$ parts, i.e., with $q=\lfloor N/c\rfloor$ and $r=N-qc$, there are  $r$ copies of $q+1$ and $c-r$ copies of $q$. Then for every $\lambda>2$,
$$
\sum_{i=1}^c \Phi_{P_{k_i}}(\lambda)\;\ge\;\sum_{i=1}^c \Phi_{P_{k_i^{\ast}}}(\lambda),
$$
with equality if and only if $(k_i)_{i=1}^c$ is a permutation of $(k_i^{\ast})_{i=1}^c$.
\end{lemma}

\begin{proof}
 If $(k_i)_{i=1}^c$ is already a permutation of $(k_i^{\ast})_{i=1}^c$, then two sums 
are equal. Otherwise, there exist indices $i\ne i'$ with $k_i\ge k_{i'}+2$. We consider the operation: Replace 
$(k_i,k_{i'})$ by $(k_i-1,k_{i'}+1)$; the sum is changed only at these 
two indices, and the constraint $k_j\ge 2$ is preserved (since 
$k_i\ge k_{i'}+2\ge 4$ gives $k_i-1\ge 3$ and $k_{i'}+1\ge 3$). Thus, the effect on 
these two indices is 
$$
\left(\Phi_{P_{k_i-1}}(\lambda) 
+\Phi_{P_{k_{i'}+1}} (\lambda)\right)-\left(\Phi_{P_{k_i}} (\lambda) +\Phi_{P_{k_{i'}}} (\lambda )\right)
= \Delta \Phi_{k_{i'}} -\Delta \Phi_{k_i-1} ,
$$ 
where $\Delta \Phi_k$ is defined as in (\ref{eq-Delta}). 
Since $k_i-1\ge k_{i'}+1>k_{i'} \ge 2$, and 
by Theorem~\ref{thm-increment} (ii), we know that 
$\Delta \Phi_k$ is strictly increasing in $k$, we have $\Delta \Phi_{k_i-1}>\Delta \Phi_{k_{i'}}$. 
So we get 
\[  \Phi_{P_{k_i-1}} (\lambda) +\Phi_{P_{k_{i'}+1}} (\lambda)  \, < \,  \Phi_{P_{k_i}} (\lambda)+\Phi_{P_{k_{i'}}} (\lambda). \]
We conclude that the operation strictly decreases the sum $\sum_{i=1}^c \Phi_{P_{k_i}}(\lambda)$.

Using the above operation repeatedly, we can turn 
any partition $(k_i)_{i=1}^c$ to a balanced partition. 
Indeed, 
we consider the quantity $Q:=\sum_{i=1}^c k_i^2$. Since $k_i\ge k_{i'}+2$, we see that 
$$
\Delta Q=(k_i-1)^2+(k_{i'}+1)^2-k_i^2-k_{i'}^2=-2(k_i-k_{i'}-1)\le-2. 
$$
 Thus, $Q$ strictly decreases under each step of the operation. 
Since $Q \ge 0$, the operation process thereby terminates at finitely many times. The terminal configuration admits no pair $(k_i,k_{i'})$ with 
$k_i\ge k_{i'}+2$, i.e., it is a balanced partition of $N$. Since each operation strictly 
decreased the sum, the starting value $\sum_{i=1}^c \Phi_{P_{k_i}} (\lambda)$ strictly exceeds the 
final value $\sum_{i=1}^c \Phi_{P_{k_i^{\ast}}} (\lambda)$.
\end{proof}

\begin{lemma}\label{lem-5.2}
Let $c\ge 1$ and $N_1<N_2$ be integers with $N_1\ge 2c$. Write $(k_j^{(i)})_{j=1}^c$ for the balanced partition of $N_i$ into $c$ parts, sorted in non-decreasing order. Then for every $j\in\{1,\ldots,c\}$,
$$
k_j^{(1)}\;\le\; k_j^{(2)}.
$$
\end{lemma}

\begin{proof}
We denote $N_i=q_i c+r_i$ with $0\le r_i<c$. The sorted balanced partition of $N_i$ is
$$
(k_j^{(i)})_{j=1}^c=(q_i,\ldots,q_i,\,q_i+1,\ldots,q_i+1), 
$$
where $q_i$ appears $c-r_i$ times and $q_i+1$ appears $r_i$ times. 
Since $N_1<N_2$ and $c$ is fixed, we have $q_1\le q_2$.
If $q_1<q_2$, then every entry of $(k_j^{(1)})$ is at most $q_1+1\le q_2$. Moreover, every entry of $(k_j^{(2)})$ is at least $q_2$. Thus, we get  $k_j^{(1)}\le q_1+1\le q_2\le k_j^{(2)}$ for every $j$.

If $q_1=q_2=q$, then $N_1< N_2$ implies $r_1<r_2$. 
In this case, $N_1$ has the partition $\{q,\ldots ,q, q+1,\ldots ,q+1\}$, where $q+1$ appears $r_1$ times, and $N_2$ has the partition $\{q,\ldots ,q, q+1,\ldots ,q+1\}$, where $q+1$ appears $r_2$ times. Since $r_1< r_2$, we see that 
$k_j^{(1)}\le k_j^{(2)}$ for all $j$. 
\end{proof}

\subsection{Proof of Theorem~\ref{main}}

Our goal is to prove that the minimizer $G\cong \overline{P_{n,s+1}}$. 
By Theorem~\ref{thm:structure}, we know that $|V(G)|=n$, and each component of $\overline G$ is either a cycle or a non-trivial path, and there are exactly $s+1$ components being paths. 
Then we denote 
$$
H_0:=\overline G \;=\; C_{\ell_1}\cup\cdots\cup C_{\ell_t}\cup P_{k_1}\cup\cdots\cup P_{k_{s+1}},
$$
where $t\ge 0$, $\ell_i\ge 3$, $k_j\ge 2$ satisfying 
 $N_c:=\sum_{i=1}^t \ell_i$ and $N_p=\sum_{j=1}^{s+1} k_j =n-N_c$. 

Applying Lemmas~\ref{lem-2.2} and~\ref{lem-Phi-cycle-path} to $G$ and to $\overline{P_{n,s+1}}$, 
we obtain
\begin{equation} \label{eq-Psi-both}
\Psi_G(\lambda)=\frac{N_c}{\lambda-2}+\sum_{j=1}^{s+1}\Phi_{P_{k_j}}(\lambda), 
\qquad
\Psi_{\overline{P_{n,s+1}}}(\lambda)=\sum_{j=1}^{s+1}\Phi_{P_{k_j^{*}}}(\lambda),
\end{equation}
where $(k_j^{*})_{j=1}^{s+1}$ is the balanced partition of $n$ into $s+1$ parts.

\begin{claim} \label{claim:1}
Path lengths in $H_0$ must be the balanced partition of $N_p$.
\end{claim}

\begin{proof}[Proof of Claim]
Suppose for the contradiction that $(k_j)_{j=1}^{s+1}$ is not a permutation of the balanced partition $(k_j^{\ast})_{j=1}^{s+1}$. Now, we construct a graph $G'$ with $\overline{G'}:=C_{\ell_1}\cup\cdots\cup C_{\ell_t}\cup P_{k_1^{\ast}}\cup\cdots\cup P_{k_{s+1}^{\ast}}$. Clearly, we have $G'\in \mathcal{G}(m)$. For every $\lambda>2$, we get  
$$
\Psi_{G'}(\lambda)-\Psi_G(\lambda)=\sum_{j=1}^{s+1}\Phi_{P_{k_j^{\ast}}}(\lambda)-\sum_{j=1}^{s+1}\Phi_{P_{k_j}}(\lambda)<0,  
$$
where the last inequality holds by using  Lemma~\ref{lem-5.1}. By the comparison principle in Theorem~\ref{thm-comparison}, we get $\rho(G')<\rho(G)$, which contradicts with the minimality of $G$. 
Hence, we conclude that $(k_j)_{j=1}^{s+1}$ is a permutation of $(k_j^{\ast})_{j=1}^{s+1}$. In other words, we have $|k_i-k_j|\le 1$ for all $i,j$.
\end{proof}

\begin{claim} \label{claim:2}
We have $t=0$.
\end{claim}

\begin{proof}[Proof of Claim]
Suppose for the contradiction that $t\ge 1$. Then $N_c=\sum_i\ell_i\ge 3t\ge 3$ and $N_p<n$. 
By Claim \ref{claim:1}, we may assume that $(k_j)_{j=1}^{s+1}$ is the balanced partition of $N_p$ sorted in non-decreasing order. 
Recall that $(k_j^{*})_{j=1}^{s+1}$ is the sorted balanced partition of $n$. Applying Lemma~\ref{lem-5.2} to $N_p< n$, we see that for every $j$, 
$$
m_j:=k_j^{*}-k_j\;\ge\; 0.
$$
Moreover, we have $\sum_{j=1}^{s+1} m_j=n-N_p=N_c\ge 3$, so at least one $m_j\ge 1$. 

We now compare $\Psi_G (\lambda)$ and $\Psi_{\overline{P_{n,s+1}}} (\lambda)$. Combining with (\ref{eq-Psi-both}), we get 
\begin{equation}\label{e-6.1}
\Psi_G(\lambda)-\Psi_{\overline{P_{n,s+1}}}(\lambda)=\frac{N_c}{\lambda-2}-\sum_{j=1}^{s+1}\bigl(\Phi_{P_{k_j^{*}}}(\lambda)-\Phi_{P_{k_j}}(\lambda)\bigr).    
\end{equation} 
By Theorem~\ref{thm-increment} (i), for every index $j$, 
we have the following two cases: 
\begin{itemize}
\item If $m_j=0$, then $\Phi_{P_{k_j^{*}}}(\lambda)-\Phi_{P_{k_j}}(\lambda)=0=\tfrac{m_j}{\lambda-2}$.

\item If $m_j\ge 1$, then $\Phi_{P_{k_j^{*}}}(\lambda)-\Phi_{P_{k_j}}(\lambda)<\tfrac{m_j}{\lambda-2}$  holds strictly.
\end{itemize}
Since at least one $m_j\ge 1$, summing over all $j$ gives
$$
\sum_{j=1}^{s+1}\bigl(\Phi_{P_{k_j^{*}}}(\lambda)-\Phi_{P_{k_j}}(\lambda)\bigr)\;<\;\sum_{j=1}^{s+1}\frac{m_j}{\lambda-2}\;=\;\frac{N_c}{\lambda-2}.
$$
Substituting into (\ref{e-6.1}), 
we obtain $\Psi_G(\lambda) >\Psi_{\overline{P_{n,s+1}}}(\lambda)$ 
for every $\lambda>2$. By the comparison principle in Theorem~\ref{thm-comparison}, we get $\rho(G)>\rho (\overline{P_{n,s+1}} )$, contradicting with the minimality of $G$.
\end{proof}

By Claims \ref{claim:1} and \ref{claim:2}, we have $t=0$, and 
the path lengths in $H_0$ form 
the balanced partition of $n$, which yields  
$H_0\cong P_{n,s+1}$, i.e. $G\cong \overline{P_{n,s+1}}$. 
This completes the proof of Theorem \ref{main}.

\section{A new proof for the 
range $s \ge \frac{1}{2}(n-2)$}

\label{sec-short-proof}

In previous section, we determined the spectral extremal graph in the difficult case for the range $1\le s < \frac{n-2}{2}$, and thereby solved Lin--Zhou's problem in Conjecture \ref{conj-LZ}. 
The range $s\ge \max\{\frac{n-6}{2} ,1\}$ was already confirmed by Lin and Zhou \cite{LZ2026} by using the Rayleigh quotient technique. 
In this range, we cannot apply the structural result stated in Theorem \ref{thm:structure}, so that we cannot use the previous lemmas for the increment analysis of $\Phi$-functions on paths and cycles. 

In this section, we provide another application of our approach developed in this paper. 
Applying the comparison principle in Theorem \ref{thm-comparison}, we give a new proof of Lin--Zhou's result for the range $s\ge \frac{n-2}{2}$ by combining $\Psi$-functions with walk enumeration. 
Counting the number of walks in a graph turns out to be powerful in studying spectral extremal problems; see, e.g., \cite{Niki2002cpc, LLZ-edge-spectral, FLLM2025}.

\smallskip 
Before comparing the $\Psi$-functions, we need to bound the diameter of $G$.

\begin{lemma} \label{lem-diam}
If $G$ is a graph on $n$ vertices with $m > \binom{n-1}{2}$ edges, then $\mathrm{diam}(G) \le 2$.
\end{lemma}

\begin{proof}
Suppose for contradiction that $\mathrm{diam}(G) \ge 3$. There exist $u, v \in V(G)$ such that distance $d_G(u, v) \ge 3$. Thus $u$ and $v$ are non-adjacent, and $N(u) \cap N(v) = \emptyset$. We now count the missing edges (non-edges) in $G$. Observe that $u$ is non-adjacent to $n - 1 - d(u)$ vertices. In addition, $v$ is non-adjacent to $d(u)$ vertices of $N(u)$. So the total number of non-edges of $u,v$ is at least $(n - 1 - d(u)) + d(u) = n - 1$. 
Then we have $e(G)\le \binom{n}{2} - (n - 1) = \binom{n-1}{2}$, contradicting with the assumption $m > \binom{n-1}{2}$. Thus, the diameter of $G$ is at most $2$.
\end{proof}

Recall that $e(G)=m={n-1 \choose 2} +s$. 
Since $s \ge 1$, we have $m > \binom{n-1}{2}$.
Lemma \ref{lem-diam} guarantees $\mathrm{diam}(G) \le 2$, which implies for any two distinct $u,v\in V(G)$,  $d_{G}(u,v)=1$ if and only if $uv\in E(G)$;  and $d_{G}(u,v)=2$ if and only if $uv\notin E(G)$. 
Let $H_0 = \overline{G}$ denote the complement of $G$. 
We conclude that $D(G)=J-I +A(H_0)$. 
Thus, the conclusions of Lemma \ref{lem-2.2} remain valid. Consequently, the $\Psi$-function of $G$ in  Definition \ref{defn} is well-defined, and Theorem \ref{thm-comparison} is still applicable. 

\smallskip 
 Note that $e := e(H_0) = \binom{n}{2} - m = n - 1 - s$. For the range $s \ge \frac{n-2}{2}$, we have $e \le  \frac{n}{2}$ and  
\[ P_{n,s+1} \cong  e P_2 \cup (n-2e) K_1. \] 
Using the Neumann series in Lemma \ref{lem-Neumann}, 
we expand the $\Psi$-function as 
\begin{equation} \label{eq-expand}
\Psi_G(\lambda) = \mathbf{1}^{\mathsf{T}}\bigl(\lambda I - A(H_0)\bigr)^{-1}\mathbf{1} = \sum_{k=0}^{\infty} \frac{\mathbf{1}^{\mathsf{T}} A(H_0)^k \mathbf{1}}{\lambda^{k+1}}.
\end{equation}
Let $w_k(H_0) := \mathbf{1}^{\mathsf{T}} A(H_0)^k \mathbf{1}$ denote the total number of walks of length $k$ in $H_0$. 

\begin{lemma} \label{lem-walk}
Let $H_0$ be a graph on $n$ vertices with $n-s-1$ edges, where $s\ge \frac{n-2}{2}$. Then $w_k(H_0) \ge w_k(P_{n,s+1})$ for every $k \ge 1$. If $H_0 \not\cong P_{n,s+1}$, then $w_k(H_0) > w_k(P_{n,s+1})$ for all $k \ge 2$.
\end{lemma} 

\begin{proof}
Note that $s\ge \frac{n-2}{2}$ and $P_{n,s+1}$ consists of some disjoint edges and isolated vertices. The isolated vertices contribute no walks, and each edge contributes exactly two walks of length $k$. So $w_k(P_{n,s+1}) = 2e$ for all $k \ge 1$. 
In graph $H_0$, every edge admits at least two trivial walks of length $k$ that use only the two endpoints.  Thus, we get $w_k(H_0) \ge 2e = w_k(P_{n,s+1})$. 
If $H_0 \not\cong P_{n,s+1}$, then there exist three vertices such that $x\sim y \sim z$, which yields many walk of length $k$ starting from $x$ (or $y$ or $z$) and alternating between $y$ and $z$. So $w_k(H_0) \ge 2e + 1 > w_k(P_{n,s+1})$ for all $k \ge 2$.
\end{proof}

Applying Lemma \ref{lem-walk}, if $H_0 \not\cong P_{n,s+1}$, then using the identity (\ref{eq-expand}) gives  
\begin{equation*}
\Psi_G(\lambda) - \Psi_{\overline{P_{n,s+1}}}(\lambda) = \sum_{k=0}^{\infty} \frac{w_k(H_0) - w_k(P_{n,s+1})}{\lambda^{k+1}}. 
\end{equation*}
Since $w_0(H_0) = w_0(P_{n,s+1}) = n$, $w_1(H_0) = w_1(P_{n,s+1}) = 2e$ and $w_k(H_0) - w_k(P_{n,s+1}) > 0$ for all $k \ge 2$, we get $\Psi_G(\lambda) > \Psi_{\overline{P_{n,s+1}}}(\lambda)$. By Theorem \ref{thm-comparison}, we obtain $\rho(G) > \rho(\overline{P_{n,s+1}})$, as desired.


\begin{thebibliography}{99}
\bibitem{Alon2021}
N. Alon,  S.M. Cioab\u{a}, B.D. Gilbert, 
J.H. Koolen, B.D. McKay, 
Addressing Johnson graphs, complete multipartite graphs, odd cycles, and random graphs,  Exp. Math., 30 (2021) 372--382.

\bibitem{ah14} 
M. Aouchiche, P. Hansen, Distance spectra of graphs: a survey, 
Linear Algebra Appl. 458 (2014) 301--386.

\bibitem{Bha1997} 
R. Bhatia, Matrix Analysis, GTM 169, Springer-Verlag, New York, 1997. 

\bibitem{BH1985}
R.A. Brualdi, A.J. Hoffman, 
On the spectral radius of $(0, 1)$-matrices,
 Linear Algebra Appl. 65 (1985) 133--146. 

\bibitem{CW2025}
Y.-J. Cheng,  C.-W. Weng, 
A matrix realization of spectral bounds, 
J. Combin. Theory Ser. B 174 (2025) 1--27. 

\bibitem{FLLM2025}
L. Fang, Y. Li, H. Lin, J. Ma, 
Spectral supersaturation for color-critical graphs, 
 (2025), arXiv:2512.22482. 

\bibitem{Fri1985}
 S. Friedland, The maximal eigenvalue of 0-1  matrices with prescribed number of ones, 
 Linear Algebra Appl. 69 (1985) 33--69.

\bibitem{GP1971}
R.L. Graham, H.O. Pollak, 
On the addressing problem for loop switching, Bell Syst. Tech. J. 50 (1971) 2495--2519.

\bibitem{GHH1977}
 R.L. Graham, A.J. Hoffman, H. Hosoya, On the distance matrix of a directed graph, 
 J. Graph Theory 1 (1977) 85--88.

\bibitem{GL1978}
R.L. Graham, L. Lovász, 
Distance matrix polynomials of trees, 
Adv. Math. 29 (1978) 60--88. 

\bibitem{LZZ2025}
S. Li, S. Zhao, L. Zou, 
Spectral extrema of graphs with fixed size: forbidden a fan graph, a friendship graph, or a theta graph, 
J. Graph Theory 110 (4) (2025) 483--495.

\bibitem{LZS2024}
X. Li, M. Zhai, J. Shu, 
A Brualdi--Hoffman--Tur\'{a}n problem on cycles, 
European J. Combin. 120 (2024), Paper No. 103966. 

 \bibitem{LLZ2024-book-4-cycle}
Y. Li, H. Liu, S. Zhang, 
More on Nosal's spectral theorem: Books and 4-cycles, J. Combin. Theory Ser. B 179 (2026) 219--249. 

\bibitem{LLZ-edge-spectral}
Y. Li, H. Liu, S. Zhang, 
An edge-spectral Erd\H{o}s--Stone--Simonovits theorem and its stability, 
 (2025), arXiv:2508.15271. 

\bibitem{LLZ-edge-color-critical}
Y. Li, H. Liu, S. Zhang, 
Edge-spectral Tur\'{a}n theorems for color-critical graphs with applications, 
 (2025), arXiv:2511.15431. 

\bibitem{LSXZ2021}
H. Lin, J. Shu, J. Xue, Y. Zhang, 
A survey on distance spectra of graphs, 
Adv. Math. (China) 50 (1) (2021) 29--76.

\bibitem{LZ2026} 
H. Lin, B. Zhou, Extremal distance spectral radius of graphs with fixed size, Adv. in Appl. Math. 173 (2026), Paper No. 102980. 

\bibitem{Niki2002cpc} 
V. Nikiforov, 
Some inequalities for the largest eigenvalue of a graph, 
Combin. Probab. Comput. 11 (2002) 179--189. 

\bibitem{Row1988}
P. Rowlinson, 
On the maximal index of graphs with a prescribed number of edges, 
Linear Algebra Appl. 110 (1988) 43--53.

\bibitem{RP1990}  
S.N. Ruzieh, D.L. Powers, The distance spectrum of the path $P_n$ and the first distance eigenvector of connected graphs, Linear Multilinear Algebra  28(1990) 75--81.

\bibitem{Saw2016}
M. Sawa,  On a symmetric representation of Hermitian matrices and its applications to graph theory, 
J. Combin. Theory Ser. B 116 (2016) 484--503.

\bibitem{Sta1987}
 R.P. Stanley, A bound on the spectral radius of graphs with $e$ edges, Linear Algebra Appl. 87 (1987) 267--269.

\bibitem{WIS2012}
S. Watanabe, K. Ishii, M. Sawa, 
A $Q$-analogue of the addressing problem of graphs by Graham and Pollak, 
SIAM J. Discrete Math. 26 (2) (2012) 527--536.

\bibitem{SI2010}
D. Stevanović, A. Ilić, Distance spectral radius of trees with fixed maximum degree, Electron. J.
Linear Algebra 20 (2010) 168--179. 


\end{thebibliography}
\end{document}